\documentclass[eqno,12pt]{article}
\usepackage{amsfonts,amssymb}
\usepackage{amssymb,amsmath,latexsym}
\usepackage{amsthm}
\usepackage[varg]{pxfonts}
\usepackage{bbm}
\usepackage{CJK}
\usepackage{fancyhdr}
\usepackage{graphicx}
\usepackage{geometry}
\usepackage{psfrag}
\usepackage{time}
\usepackage{colortbl}
\usepackage{stmaryrd}
\usepackage{mathrsfs}

\usepackage{indentfirst,latexsym,bm}

\makeatletter \@addtoreset{equation}{section}

\newtheorem{prop}{Proposition}[section]
\newtheorem{thm}[prop]{Theorem}
\newtheorem{defn}{Definition}[section]

\newtheorem{examp}{Example}[section]
\newtheorem{rem}{Remark}[section]

\begin{document}

\baselineskip 5true mm
\begin{center}
 {\Large\bf  On a type of semi-sub-Riemannian connection on a sub-Riemannian manifold}
\end{center}
\centerline{ Yanling Han \footnote{Supported by a Grant-in-Aid for
Scientific  Research from Nanjing University of Science and
Technology (KN11008). } }
\begin{center}
{\scriptsize  Dept. of Applied Mathematics, Nanjing  University of Science and Technology, Nanjing 210094, P. R. China\\
School of Science, Shandong Polytechnic University, Jinan 250353,
P.R.China\\
 E-mail: hanyanling1979@163.com}
\end{center}
\centerline{ Peibiao Zhao \footnote{Supported by a Grant-in-Aid for
Science Research from Nanjing University of Science and Technology
(2011YBXM120), by NUST Research Funding No.CXZZ11-0258,AD20370 and
by NNSF (11071119). }}
\begin{center}
{\scriptsize  Dept. of Applied Mathematics, Nanjing  University of Science and Technology, Nanjing 210094, P. R. China\\
E-mail: pbzhao@njust.edu.cn}
\end{center}

\begin{center}
\begin{minipage}{128mm}
{\footnotesize {\bf Abstract}   The authors first in this paper
define a semi-symmetric metric non-holonomic connection (called in
briefly a semi-sub-Riemannian connection) on sub-Riemannian
manifolds, and study the relations between sub-Riemannian
connections and semi-sub-Riemannian connections. An invariant under
a connection transformation $\nabla\rightarrow D$ is obtained. The
authors then further deduce a sufficient and necessary condition
that a sub-Riemannian manifold associated with a semi-sub-Riemannian
connection is flat, and derive that a sub-Riemannian manifold with
vanishing curvature with respect to semi-sub-Riemannnian connection
$D$ is a group manifold
if and only if it is of constant curvature.
 \\
{\bf MR(2000)} 53C20, 53D11.}
\end{minipage}
\end{center}

\section{ Introduction}

 The study of transformation in Riemannian geometry has experienced a long
 time. In 1924, A. Fridmann and J. A. Schouten \cite{FS2} first introduced
the concept of a semi-symmetric linear connection in a differential
manifold, namely, a linear connection $\tilde{\nabla}$ is said to be
a semi-symmetric connection if its torsion tensor $\tilde{T}$ is of
the form
\begin{eqnarray*}
\tilde{T}(X,Y)=\pi(Y)X-\pi(X)Y, \forall X,Y\in\Gamma(TM),
\end{eqnarray*}
where $\pi$ is of 1-form associated with vector P on M, and P is
defined by $g(X,P)=\pi(X)$. In 1970, K. Yano \cite{Ya} considered a
semi-symmetric metric connection (that means a linear connection is
both metric and semi-symmetric) on a Riemannian manifold and studied
some of its properties. He pointed out that a Riemannian manifold is
conformal flat if and only if it admits a semi-symmetric metric
connection whose curvature tensor vanishes identically. He also
proved that a Riemannian manifold is of constant curvature if and
only if it admits a semi-symmetric metric connection for which the
manifold is a group manifold, where a group manifold is a
differential manifold admitting a linear connection $\tilde{\nabla}$
such that its curvature tensor $\tilde{R}$ vanishes and the
covariant derivative of torsion tensor $\tilde{T}$ with respect to
$\tilde{\nabla}$ is vanishing. Liang in his paper \cite{Li}
discussed some properties of semi-symmetric metric connections and
proved that the projective curvature tensor with respect to
semi-symmetric metric connections coincides with the projective
curvature tensor with respect to Levi-civita connection if and only
if the characteristic vector is proportional to a Riemannian metric.
The authors \cite{ZSY} introduced the concept of the projective
semi-symmetric metric connection, found an invariant under the
transformation of projective semi-symmetric connections and
indicated that this invariant could degenerate into the Weyl
projective curvature tensor under certain conditions, so the Weyl
projective curvature tensor is an invariant as for the
transformation of the special projective semi-symmetric connection.
For the study of semi-symmetric metric connections, the authors have
other interesting results \cite{FYZ,Zh1,Zh2,Zh3,Zh4,ZL}. Recently,
the authors in paper \cite{ZJ} even studied the theory of
transformations on Carnot Caratheodory spaces, and obtained the
conformal invariants and projective invariants on
Carnot-Caratheodory spaces with the view of Felix Klein.

In 1990, N. S. Agache and M. R. Chafle \cite{AC}  discussed a
semi-symmetric non-metric connection on a Riemannian manifold. A
semi-symmetric connection $\tilde{\nabla}$ is said to be a
semi-symmetric non-metric connection if it satisfies the conditions:
\begin{eqnarray*}
\tilde{\nabla}_{X}Y=\nabla_{X}Y+\pi(Y)X+g(X,Y)P,~\forall X, Y,
Z\in\Gamma(TM),
\end{eqnarray*}
\begin{eqnarray*}
\tilde{\nabla}_{Z}g(X,Y)=-2\pi(X)g(Y,Z)-2\pi(Y)g(X,Z),~\forall X, Y,
Z\in\Gamma(TM),
\end{eqnarray*}
where $\nabla$ is Levi-civita connection. This semi-symmetric
non-metric connection was further developed by U. C. De and S. C.
Biswas \cite{DB}, U. C. De and D. Kamily \cite{DK}. N. S. Agashe and
M. R. Chafle \cite{AC} defined the curvature tensor with respect to
semi-symmetric non-metric connections, and proved the Weyl
projective curvature tensor with respect to semi-symmetric
non-metric connections is equal to the Weyl projective curvature
tensor with respect to Levi-Civita connection. They further got a
necessary and sufficient condition that  a Riemannian manifold with
vanishing Ricci tensor with respect to semi-symmetric non-metric
connections being projectively flat if and only if the curvature
tensor with respect to semi-symmetric non-metric connections is
vanished. U. C. De and S. C. Biswas \cite{DB} discussed the
semi-symmetric non-metric connection on Riemannian manifold by using
the similar arguments, and obtained some properties of curvature
tensors with respect to semi-symmetric non-metric connections, and
proposed that two semi-symmetric non-metric connections would be
equal under certain conditions.

The study of geometric analysis in sub-Riemannian manifolds has been
an active field over the past several decades. The past decade has
witnessed a dramatic and widespread expansion of interest and
activity in sub-Riemannian geometry. In particular, round about
1993, since the formidable papers were published in succession,
these works stimulate such research fields to present a scene of
prosperity, and demonstrate the abnormal importance of this topic.
Sub-Riemannian manifolds, on the one hand, are the natural
development of Riemannian manifolds, and are the basic metric spaces
on which one can consider the problems of geometric analysis; On the
other hand, sub-Riemannian manifolds have been found useful in the
study of theories and applications of Control theory, PDEs, Calculus
of Variations, Mechanic, Gauge fields, etc. The study of geometric
analysis in sub-Riemannian manifolds is carrying on the following
two folds. The first fold is describing the geometric properties of
sub-Riemannian manifolds\cite{B,DGN,FSC,GN}; The second fold is
devoted to the analysis problem of Sub-Riemannian
manifolds\cite{CL,JX,LS}. In the past decades, we have focused our
attention on the sub-Riemannian geodesics, and  got some interesting
and remarkable results. Although a sub-Riemannian manifold is an
natural generalization of a Riemannian manifold,  there are some
essential differences. One of the essential differences is that
there exists a kind of strange geodesics which are minimal geodesics
and topological stability, but does not satisfy the geodesics
equation. We call them singular geodesics. The existence of singular
geodesics shows the importance of sub-Riemannian geometry. The
second difference is that the endpoint mapping  can be defined by
the normal sub-Riemmanian geodesic but it is not diffeomorphic any
more. On the other hand, the horizontal connection
$\nabla^{H,\Sigma}$, used for instance for studying the minimal
surface and isoperimetric problem in sub-Rieamnnian manifolds,
defined on hypersurface $\Sigma$ is in general not torsion free, and
therefore it is not Levi-Civita any more, so the horizontal second
fundamental form $II^{H,\Sigma}$ is not symmetric, which is also
different from Riemannian case\cite{DGN2}. In this paper we will
take the liberty of considering the geometries of sub-Riemannian
manifolds via a point of view of transform groups, our final purpose
is to establish the relevant geometries in the sense of
transformative theories.

As it is well known, there exists a unique symmetric metric
nonholonomic connection (i.e. sub-Riemannian connection or
horizontal connection in this paper) in sub-Riemannian manifolds
just as Levi-Civita connection in Riemannian manifolds. According to
the geometric characteristics of Levi-Civita connection, this
symmetric metric nonholonomic connection in sub-Riemannian manifolds
can preserve the inner product of any two horizontal vector fields
when they transport along a horizontal curve. However there may be
existing a bad nonholonomic connetion in a sub-Riemannian manifold
which can not preserve the torsion property, so it is urgent and
important to study a kind of nonholonomic metric connection that is
not symmetric. The  problem of geometries and analysis of a
semi-symmetric metric nonholonomic connection emerges as the times
require.
 The semi-symmetric
nonholonomic metric connection in this paper is just a special
non-symmetric nonholonomic connection. Taking into account that
sub-Riemannian manifolds are a natural generalization of Riemannian
manifolds, we would ask whether we can consider the invariants from
symmetric metric nonholonomic connections to semi-symmetric metric
nonholonomic connection. Once we found the invariants under
connection transformations, we could study the property of an object
connection through an original connection. In order to study the
geometric properties in sub-Riemannian manifold, the second author
first discussed the transformations in Carnot-Caratheodory spaces,
and got the conformal invariants and projective invariants, which
can be regarded as an natural generalization of those conclusions in
Rimennian manifolds. We in this paper wish to use the unique
nonholonomic connection to solve the posed problems above. To the
author$^{,}$s knowledge, the study of the semi-symmetric metric
connection in sub-Riemannian manifolds is still a gap.

In this paper, we first define a semi-symmetric metric non-holonomic
connection in sub-Riemannian manifolds, and derive the relations
between a symmetric metric non-holonomic connection and a
semi-symmetric metric non-holonomic connection, and get an invariant
under the connection transformation $\nabla\rightarrow D$. We
further define the Weyl conformal curvature tensor
$\bar{C}^{h}_{ijk}$ and the Weyl projective curvature tensor
$\bar{W}^{h}_{ijk}$ of semi-symmetric metric nonholonomic
connections, and find that $\bar{C}^{h}_{ijk}$ is no longer an
invariant under the connection transformation from $\nabla$ to $D$,
which is obviously different from the Riemannian case. On the other
hand, we also deduce a sufficient and necessary condition that a
sub-Riemannian manifold admitting semi-symmetric metric connection
is flat. At last, we consider a group manifold and find the Carnot
group is an example of group manifolds, at the same time, we prove
that, a sub-Riemannian manifold associated with a semi-symmetric
metric connection  is a group manifold if and only if  the
sub-Riemannian manifold is of constant curvature.

The organization of this paper is as follows. In section 2, we will
recall and give the necessary information about Schouten curvature
tensor and symmetric metric connection in sub-Riemannian manifold.
Section 3 is devoted to the new definition and main Theorems.

\section{ Preliminaries }
\setcounter{section}{2}
 \setcounter{equation}{0}
 Let $M^{n}$ be an
$n$-dimensional smooth manifold. For each point $p\in M^{n}$, there
assigns a $\ell(2<\ell<n)$-dimensional subspace $V^{\ell}(p)$ of the
tangent space $T_{p}M$, then $V^{\ell}=\bigcup_{p\in M}V^{\ell}(p)$
forms a tangent sub-bundle of tangent bundles $TM=$ $\bigcup_{p\in
M}T_{p}M,$ $V^{\ell}$ is called a $\ell$-dimensional distribution
over $M^{n}$. For any point $p$, if there exists a neighbourhood $U$
and $\ell$ linearly independent vector fields $X_{1},$ $\cdots$,
$X_{\ell}$ in $U$ such that for each point $q\in U,$ $X_{\ell}(q),$
$\cdots$, $X_{\ell}(q)$ is a basis of subspace $V^{\ell}(q)$, then
we call  $V^{\ell}$ the $\ell$-dimensional smooth distribution
(called also a horizontal bundle), and $X_{1},$ $\cdots, X_{\ell}$
are called a local basis of $V^{\ell}$ in $U$. We also say that
$X_{1},$ $\cdots, X_{\ell}$  generate  $V^{\ell}$ in $U$. We denote
by $V^{\ell}|_{U}=Span\{X_{1}, \cdots, X_{\ell}\}$.

\begin{defn}\label{def21} We call $(M, V_{0}, g)$
a sub-Riemannian manifold with the sub-Riemannian structure
$(V_{0},g)$, if $V_{0}$ is a $\ell$-dimensional smooth distribution
over $M^{n}$, and $g$ is a fibre inner product in $V_{0}$. Here $g$
is called a sub-Riemannian metric and $V_{0}$ is called a horizontal
bundle. In general, $g$ can be regarded as some Riemannian metric
$\langle\cdot,\cdot\rangle$, defined on tangent bundle $TM$,
restricted to $V_{0}$. \end{defn}

Throughout the paper, we denote by $\Gamma(V_{0})$ the
$C^{\infty}(M)$ -module of smooth sections on $V_{0}$. Also, if not
stated otherwise, we use the following ranges for indices:
$i,j,k,h,\cdots \in\{1,\cdots,\ell\}$, $\alpha,\beta,\cdots
\in\{\ell+1,\cdots,n\}$. The repeated indices with one upper index
and one lower index indicates summation over their range.

\begin{defn}\label{def22} A nonholonomic connection
on sub-bundle $V_0\subset TM$ is a binary mapping $\nabla$ :
$\Gamma(V_{0})\times\Gamma(V_{0})\rightarrow\Gamma(V_{0})$
satisfying the following:
\begin{equation}\label{2}
 \nabla_{X}(Y+Z)=\nabla_{X}Y+\nabla_{X}Z,
\end{equation}
\begin{equation}\label{2}
\nabla_{X}(fY)=X(f)Y+f\nabla_{X}Y,
\end{equation}
\begin{equation}\label{2}
\nabla_{fX+gY}Z=f\nabla_{X}Z+g\nabla_{Y}Z,
\end{equation}
 where $X,$ $Y,$ $Z\in\Gamma(V_{0}),$ $f,$ $g\in C^{\infty}(M)$.
 \end{defn}

In order to study the geometry of $\{M, V_{0}, g\}$, we suppose that
there exists a Rimannian metric $<\cdot,\cdot>$ and $V_{1}$ is taken
as the complementary orthogonal distribution to $V_{0}$ in $TM$,
then, there holds $V_{0}\oplus V_{1}=TM$. Here we call $V_{1}$ the
vertical distribution. Denote by $X_{0}$  the projection of the
vector field $X$ from $TM$ onto $V_{0},$  and by $X_{1}$  the
projection of the vector field $X$ from $TM$ onto $V_{1}$.

\begin{defn}\label{def23} The torsion tensor of
nonhholonomic connection $\nabla$ is defined by
\begin{equation}\label{2}
T(X,Y)=\nabla_{X}Y-\nabla_{Y}X-[X,Y]_{0},\forall X,Y\in
\Gamma(V_{0}).
\end{equation}
\end{defn}
From Definition \ref{def23} we know the torsion tensor of horizontal
vector fields is still horizontal vector field, so we call it the
horizontal torsion tensor.

Assume that $\{e_{i}\},i=1,\cdots,\ell$ is a basis of $V_{0}$, then
the formulas $\nabla_{e_{i}}e_{j}= \{_{ij}^{k}\}e_{k},$ $i, j, k=1,$
$\cdots, \ell$ define $\ell^{3}$ functions as $\{_{ij}^{k}\}$, we
call $\{_{ij}^{k}\}$ the connection coefficients of the
non-holonomic connection $\nabla$.

It is well known that the Lie bracket $[\cdot, \cdot]$ on $M$ is a
Lie algebra structure of smooth tangent vector fields $\Gamma(TM)$,
then it is easy to see that   the following formula
\begin{equation}\label{2}
[e_{i}, e_{j}]_{0}=\Omega_{ij}^{k}e_{k}\nonumber,\\
\end{equation}
 determine $\ell^{3}$ functions $\Omega_{ij}^{k}$.

About the existence of this class of connections defined on the
horizontal bundle $V_{0}$, we have the same result as Riemannian
case.
\begin{thm}\label{thm21}\cite{CLG,TY}
Given a sub-Riemannian manifold $(M, V_{0}, g)$, then there exists a
unique nonholonomic connection satisfying
\begin{equation}\label{2}
Zg(X,Y)=g(\nabla_{Z}X,Y)+g(X,\nabla_{Z}Y),
\end{equation}
\begin{equation}\label{2}
T(X,Y)=\nabla_{X}Y-\nabla_{Y}X-[X,Y]_{0}=0.
\end{equation}
\end{thm}
\begin{rem}\label{re21}
Similar to Riemannian manifolds, we also say that the non-holonomic
connections with property $(2.5)$ and $(2.6)$ are metric and
torsion-free, respectively. An non-holonomic connection satisfying
$(2.5)$ and $(2.6)$ is called a sub-Riemannian connection or a
horizontal connection. For a simplified proof of Theorem
\ref{thm21}, one can see \cite{ZJ} for details. On the other hand,
K. Yano \cite{Ya} posed a proof with a method of projecting the
Riemannian connection onto the distribution to derive Theorem
\ref{thm21} in the case of Riemannian manifolds.
\end{rem}

Next we discuss the horizontal connection of Carnot group, which is
a very important example of sub-Riemannian manifolds. If $G$ is Lie
group with graded Lie algebra satisfying
\begin{eqnarray}\label{2}
&&\hbar=V_{0}\oplus V_{1}\oplus\cdots \oplus V_{r-1}\nonumber,\\
&&[V_{0},V_{j}]=V_{j+1}, j=1,2,\cdots,r-1,
\end{eqnarray}
then we call $G$ a Carnot group. Let $\circ$ be the group law on
$G$, then the left translation operator is $L_{p}:q\rightarrow
p\circ q$, denote by $(L_{p})_{*}$ the differential of $L_{p}$. Now
we can define the horizontal subspace as
\begin{equation}\label{2}
HG_{p}=(L_{p})_{*}(V_{0}),\nonumber\\
\end{equation}
for any point $p\in G$,  and the horizontal bundle as
\begin{equation}\label{2}
HG=\bigcup_{p\in G}HG_{p}\nonumber.\\
\end{equation}
Then we further consider the vertical distribution on $G$ defined by
\begin{equation}\label{2}
VG_{p}=(L_{p})_{*}(V_{1}\oplus\cdots \oplus V_{r-1}),\nonumber\\
\end{equation}
\begin{equation}\label{2}
VG=\bigcup_{p\in G}VG_{p}\nonumber.\\
\end{equation}
Now, we fix a basis $X_{1},\cdots,X_{\ell}$ formed by the left
invariant vector fields, then, by $(2.7)$, we deduce that
\begin{equation}\label{2}
[\Gamma(VG),X_{k}]\in \Gamma(VG),
\end{equation}
 and fix the
inner product $<\cdot,\cdot>$ in $TG$ such that the system of
left-invariant vector fields
$\{X_{1},\cdots,X_{k},Y_{1},\cdots,Y_{n-k}\}$ is an orthnormal basis
of $TG$, so there is an natural nonholonomic connection $\nabla$ on
$HG$ satisfying
\begin{equation}\label{2}
\nabla_{X}Y=X(Y^{i})X_{i},
\end{equation}
where $Y=Y^{i}X_{i}$.

 For sub-Riemannian manifolds, J. A. Shouten first considered
the curvature problem of non-holonomic connections(see \cite{CLG}),
he defined a curvature tensor as follows:
 \begin{defn}\label{def24}
A Shouten curvature tensor is a mapping $K$ :
$\Gamma(V_{0})\times\Gamma(V_{0})\rightarrow\Gamma(V_{0})$ defined
by
\begin{equation}\label{2}
K(X,Y)Z=\nabla_{X}\nabla_{Y}Z-\nabla_{Y}\nabla_{X}Z-\nabla_{[X,Y]_{0}}Z-[[X,Y]_{1},Z]_{0},
\end{equation}
where $X,$ $Y,$ $Z\in\Gamma(V_{0})$.\end{defn}

 If $M$ is a Carnot group $G$, the
Schouten curvature tensor, because of (2.8), is of the form
\begin{equation}\label{2}
K(X,Y)Z=\nabla_{X}\nabla_{Y}Z-\nabla_{Y}\nabla_{X}Z-\nabla_{[X,Y]_{0}}Z.
\end{equation}

\begin{rem}\label{re22}
It is easy to check that Definition \ref{def24} is well defined. In
fact, we know that the following formulas are tenable.
\begin{equation}\label{2}
K(fX, Y)Z=fK(X, Y)Z\nonumber,\\
\end{equation}
\begin{equation}\label{2}
K(X, fY)Z=fK(X, Y)Z\nonumber,\\
\end{equation}
\begin{equation}\label{2}
K(X, Y)(fZ)=fK(X, Y)Z\nonumber,\\
\end{equation}
\end{rem}

For Shouten tensor, by using Jacobi identity of Poisson bracket and
Definition \ref{def24}, we have
\begin{equation}\label{2}
K(X, Y)Z=-K(Y, X)Z,
\end{equation}
\begin{equation}\label{2}
K(X,Y)Z+K(Y,Z)X+K(Z,X)Y=0.
\end{equation}

It is well known that there hold the following formulas for the
curvature tensor $R$ over Riemannian manifolds
\begin{equation}\label{2}
R(X,Y,Z,W)=-R(Y,X,Z,W),
\end{equation}
\begin{equation}\label{2}
R(X,Y,Z,W)=-R(X,Y,W,Z),
\end{equation}
\begin{equation}\label{2}
R(X,Y,Z,W)=R(Z,W,X,Y).
\end{equation}

We also define (0,4)-tensor by $K(X,Y,Z,W)=g(K(X,Y)Z,W)$, which
satisfies the following
\begin{equation}\label{2}
K(X,Y,Z,W)=-K(Y,X,Z,W),
\end{equation}
\begin{equation}\label{2}
K(X,Y,Z,W)+K(Y,Z,X,W)+K(Z,X,Y,W)=0.
\end{equation}

However, since the horizontal distribution $V_{0}$ is not
involutive, so the curvature tensor $K$ does not satisfy
$K(X,Y,Z,W)=-K(X,Y,W,Z)$, we only obtain
\begin{eqnarray*}\label{2}
K(X,Y,Z,W)&=&-K(X,Y,W,Z)-g([[X,Y]_{1},W]_{0},Z)-g([[X,Y]_{1},W]_{0},Z)\\
&+&[X,Y]_{1}g(Z,W).
\end{eqnarray*}
 When $V_{0}$ is involutive, i.e., $[X,Y]_{1}=0$, in this
setting, we have the analogue similar to Riemannian curvature
tensors.
\begin{rem}\label{re23}
Since the curvature tensor $K$ does not satisfy properties
$(2.15),(2.16)$, so we can not give out the second Bianchi identity
of Shouten curvature tensors similar to Riemannian curvature
tensors.
\end{rem}

 Let $\{e_{i}\}$ be a  basis  of $V_{0}$, we denote by
\begin{eqnarray*}\label{2}
&&K(e_{i},e_{j})e_{k}=K^{h}_{ijk}e_{h},\\
&&\nabla_{e_{i}}e_{j}=\{_{ij}^{k}\}e_{k},\\
&&[e_{i},e_{j}]_{0}=\Omega_{ij}^{k}e_{k},\\
&&[e_{i},e_{j}]_{1}=M_{ij}^{\alpha}e_{\alpha},\\
&&[[e_{i},e_{j}]_{1},e_{k}]_{0}=M_{ij}^{\alpha}\Lambda_{\alpha
k}^{h}e_{h}.
\end{eqnarray*}

Then we know that
\begin{eqnarray}\label{211}
K^{h}_{ijk}=e_{i}(\{_{jk}^{h}\})-e_{j}(\{_{ik}^{h}\})+\{_{jk}^{e}\}\{_{ie}^{h}\}
-\{_{ik}^{e}\}\{_{je}^{h}\}-\Omega_{ij}^{e}\{_{ke}^{h}\}
-M_{ij}^{\alpha}\Lambda_{\alpha k}^{h}
\end{eqnarray}
Since $\nabla$ is torsion free, then we get
 \begin{eqnarray*}\label{211}
\nabla_{e_{i}}e_{j}-\nabla_{e_{j}}e_{i}-[e_{i},e_{j}]_{0}=0,
\end{eqnarray*}
so we arrive at
 \begin{equation}\label{211}
\{_{ij}^{k}\}-\{_{ji}^{k}\}=\Omega_{ij}^{k},
\end{equation}
  we further have
\begin{equation}\label{211}
[e_{i}, e_{j}]-\Omega_{ij}^{k}e_{k}=M_{ij}^{\alpha}e_{\alpha}.
\end{equation}
 Especially, if the horizontal distribution
$V_{0}$ is involutive, then we obtain
 \begin{eqnarray}\label{211}
K^{h}_{ijk}=e_{i}(\{_{jk}^{h}\})-e_{j}(\{_{ik}^{h}\})+\{_{jk}^{e}\}\{_{ie}^{h}\}
-\{_{ik}^{e}\}\{_{je}^{h}\}-\Omega_{ij}^{e}\{_{ke}^{h}\}.
\end{eqnarray}

 In this basis, $(2.12),(2.13)$ can be rewritten,  respectively, as
\begin{equation}\label{2}
K^{h}_{ijk}=-K^{h}_{jik},
\end{equation}
\begin{equation}\label{2}
K^{h}_{ijk}+K^{h}_{jki}+K^{h}_{kij}=0,
\end{equation}
We call $(2.13), (2.18)$ and $(2.24)$ the first Bianchi identity of
sub-Riemannian connection $\nabla$.

In $(2.24)$, by taking $j=h=e$ and using $(2.23)$, we get
\begin{equation}\label{2}
K^{e}_{kie}=K^{e}_{kei}-K^{e}_{iek},
\end{equation}
It is clear that $K^{e}_{kie}$ is an anti-symmetric (0,2) tensor ,
which is different from Riemannian case. So
\begin{eqnarray*}\label{2}
0=K^{e}_{kie}g^{ki}+K^{e}_{ike}g^{ki}=K^{e}_{kie}g^{ki}+K^{e}_{ike}g^{ik}=2K^{e}_{kie}g^{ki}.
\end{eqnarray*}
Now multiplying $g^{ki}$ at both side of $(2.25)$, then
$g^{ki}K^{e}_{kei}-K^{e}_{iek}g^{ki}=0$. Similar to the case of
Riemannian manifolds, we call $K=g^{ik}K^{e}_{iek}$ the scalar
curvature of Shouten curvature tensors.

 \section{ Main Theorems and Proofs} \setcounter{section}{3}
 \setcounter{equation}{0}

 Theorem \ref{thm21} shows that there exists unique metric and torsion free nonholonomic
 connection in sub-Riemannian manifolds, while there also exist other some nonholonomic
 connections which is not compatible with sub-Riemannian metric any more, nor is torsion
 free. For the first time, we introduce a very important nonholonomic
 connection-semi-sub-Riemannian connection. Roughly speaking, a semi-sub-Riemannian connection is a nonholonomic
 connection with non-vanishing torsion tensor which is compatible with sub-Riemannian
 metric. More precisely, let $D$ be another non-holonomic connection on M and the
coefficients be $\Gamma_{ij}^{k}$. D is said to be a  metric
connection if it satisfies
\begin{equation}\label{2}
(D_{Z}g)(Y,Z)=Zg(X,Y)-g(D_{Z}X,Y)-g(X,D_{Z}Y)=0, \forall
X,Y,Z\in{V_0},
\end{equation}
Now we give a new definition below

 \begin{defn}\label{def31} A nonholonomic connection is called a
 semi-sub-Riemannian connection, if it is metric and it$^{,}$s torsion tensor satisfies
\begin{equation}\label{2}
T(X,Y)=D_{X}Y-D_{X}Y-[X,Y]_{0}=\pi(Y)X-\pi(Y)X ,   \forall
X,Y,Z\in{V_{0}},
\end{equation}
where $\pi$ is a smooth 1-form.\end{defn}

For the semi-sub-Riemannian connection $D$, recurrent $X,Y,Z\in
V_{0}$ in $(3.1)$, and by a direct computation, we get
\begin{equation}\label{2}
D_{X}Y=\nabla_{X}Y+\pi(Y)X-g(X,Y)P,
\end{equation}
where $P$ is a vector field defined by $g(P,X)=\pi(X)$.

\begin{rem}
$(3.3)$ is also called semi-symmetric connection transformation of
$\nabla$. It is easy to check the semi-symmetric connection
transformation of metric torsion-free nonholonmice connection is
still a metric connection by $(3.3)$. This transformation will
change horizontal curves into horizontal curves, however it is not
true for the horizontal curves paralleling with itself(i.e. normal
geodesics), we will discuss the connection transformations that
conserve the normal geodesics in forthcoming papers.
\end{rem}

In local  frame $\{e_{i}\}$, denote by $\pi(e_{i})=\pi_{i}$,
$\pi^{i}=g^{ij}\pi_{j}$, then we know
\begin{equation}\label{2}
\Gamma_{ij}^{k}=\{_{ij}^{k}\}+\delta_{i}^{k}\pi_{j}-g_{ij}\pi^{k},
\end{equation}
we define the Schouten curvature tensor of semi-sub-Riemannnian
connection $D$ is
\begin{eqnarray}\label{211}
R^{h}_{ijk}=e_{i}(\Gamma_{jk}^{h})-e_{j}(\Gamma_{ik}^{h})+\Gamma_{jk}^{e}\Gamma_{ie}^{h}
-\Gamma_{ik}^{e}\Gamma_{je}^{h}-\bar{\Omega}_{ij}^{e}\Gamma_{ke}^{h}
-\bar{M}_{ij}^{\alpha}\bar{\Lambda}_{\alpha k}^{h},
\end{eqnarray}
where
\begin{equation}\label{2}
[e_{i},e_{j}]_{0}=\bar{\Omega}_{ij}^{k}e_{k},\nonumber\\
\end{equation}
\begin{equation}\label{2}
[e_{i},e_{j}]_{1}=\bar{M}_{ij}^{\alpha}e_{\alpha},\nonumber\\
\end{equation}
\begin{equation}\label{2}
[[e_{i},e_{j}]_{1},e_{k}]_{0}=\bar{M}_{ij}^{\alpha}\bar{\Lambda}_{\alpha k}^{h}e_{h},\nonumber\\
\end{equation}
then by using (2.20), (2.21) and (3.4), we have
\begin{eqnarray}\label{2}
\begin{cases}
\bar{\Omega}_{ij}^{k}=\Omega_{ij}^{k}\\
\bar{M}_{ij}^{\alpha}=M_{ij}^{\alpha}\\
 \bar{\Lambda}_{\alpha
k}^{h}=\Lambda_{\alpha k}^{h}
\end{cases}
\end{eqnarray}
Substituting $(3.4)$ and $(3.6)$ into $(3.5)$ and by straightway
computation, we can get the relation between the Schouten curvature
tensor of $D$ and $\nabla$ as follows
\begin{equation}\label{2}
R^{h}_{ijk}=K^{h}_{ijk}+\delta_{j}^{h}\pi_{ik}-\delta_{i}^{h}\pi_{jk}+\pi_{j}^{h}g_{ik}-\pi_{i}^{h}g_{jk},
\end{equation}
where
\begin{equation}
\pi_{ik}=\nabla_{i}\pi_{k}-\pi_{i}\pi_{k}+\frac{1}{2}g_{ik}\pi_{h}\pi^{h},
\end{equation}
\begin{equation}
\pi_{i}^{j}=\pi_{ik}g^{jk}=\nabla_{i}\pi^{j}-\pi_{i}\pi^{j}+\frac{1}{2}\delta_{i}^{j}\pi_{h}\pi^{h},
\end{equation}
\begin{equation}
\nabla_{i}\pi_{j}=e_{i}(\pi_{j})-\{^{k}_{ij}\}\pi_{k}.
\end{equation}
Here we call $\pi_{ij}$ the characteristic tensor of $D$, and
$\alpha=\pi_{ij}g^{ij}=\pi_{i}^{i}$. Contracting $j$ and $h$ in
$(3.7)$, we have
\begin{equation}\label{2}
R^{e}_{iek}=K^{e}_{iek}+(\ell-2)\pi_{ik}+\alpha g_{ik}.
\end{equation}
Multiplying $(3.11)$ by $g^{ik}$ we get
\begin{equation}\label{3}
R=K+2(\ell-1)\alpha,
\end{equation}
so there is
\begin{equation}\label{3}
\alpha=\frac{R-K}{2(\ell-1)}.
\end{equation}
Substituting $(3.13)$ into $(3.11)$ we have
\begin{equation}
\pi_{ik}=\frac{1}{\ell-2}(R^{e}_{iek}-K^{e}_{iek}-\frac{R-K}{2(\ell-1)}g_{ik}),
\end{equation}
\begin{equation}
\pi_{i}^{h}=\frac{1}{\ell-2}\{(R^{e}_{iek}-K^{e}_{iek})g^{kh}-\frac{R-K}{2(\ell-1)}\delta_{i}^{h}\},
\end{equation}
then substituting $(3.14), (3.15)$ into $(3.7)$,  we get
\begin{eqnarray}
&&R^{h}_{ijk}-\frac{1}{\ell-2}\{\delta_{j}^{h}(R^{e}_{iek}-\frac{R}{2(\ell-1)}g_{ik})-\delta_{i}^{h}(R^{e}_{jek
}-\frac{R}{2(\ell-1)}g_{jk})
\nonumber\\
&-&     g_{ik}(R^{e}_{jef}g^{fh}-\frac{R}{2(\ell-1)}\delta_{j}^{h}
)+g_{jk}(R^{e}_{ief}g^{fh}-\frac{R}{2(\ell-1)}\delta_{i}^{h} )\}
\nonumber\\
&=&K^{h}_{ijk}-\frac{1}{\ell-2}\{\delta_{j}^{h}(K^{e}_{iek}-\frac{R}{2(\ell-1)}g_{ik})-\delta_{i}^{h}(K^{e}_{jek
}-\frac{K}{2(\ell-1)}g_{jk})
\nonumber\\
&+&     g_{ik}(K^{e}_{jef}g^{fh}-\frac{K}{2(\ell-1)}\delta_{j}^{h}
)-g_{jk}(K^{e}_{ief}g^{fh}-\frac{K}{2(\ell-1)}\delta_{i}^{h} )\}.
\end{eqnarray}

Let
\begin{eqnarray*}
\bar{S}^{h}_{ijk}&=&R^{h}_{ijk}-\frac{1}{\ell-2}\{\delta_{j}^{h}(R^{e}_{iek}-\frac{R}{2(\ell-1)}g_{ik})-\delta_{i}^{h}(R^{e}_{jek
}-\frac{R}{2(\ell-1)}g_{jk})\\
&-&
 g_{ik}(R^{e}_{jef}g^{fh}-\frac{R}{2(\ell-1)}\delta_{j}^{h})+g_{jk}(R^{e}_{ief}g^{fh}-\frac{R}{2(\ell-1)}\delta_{i}^{h}
)\}\\
&=&R^{h}_{ijk}-\frac{1}{\ell-2}\{\delta_{j}^{h}R^{e}_{iek
}-\delta_{i}^{h}R^{e}_{jek}+g_{ik}g^{fh}R^{e}_{jef}-g_{jk}g^{fh}R^{e}_{ief}\}\\
&+&
\frac{R}{(\ell-1)(\ell-2)}(g_{ik}\delta_{j}^{h}-g_{jk}\delta_{i}^{h}),
\end{eqnarray*}
\begin{eqnarray}\label{211}
S^{h}_{ijk}&=&K^{h}_{ijk}-\frac{1}{\ell-2}\{\delta_{j}^{h}(K^{e}_{iek}-\frac{K}{2(\ell-1)}g_{ik})-\delta_{i}^{h}(K^{e}_{jek
}-\frac{K}{2(\ell-1)}g_{jk})
\nonumber\\
&+&    g_{ik}(K^{e}_{jef}g^{fh}-\frac{K}{2(\ell-1)}\delta_{j}^{h}
)-g_{jk}(K^{e}_{ief}g^{fh}-\frac{K}{2(\ell-1)}\delta_{i}^{h} )\}\nonumber\\
&=&K^{h}_{ijk}-\frac{1}{\ell-2}\{\delta_{j}^{h}K^{e}_{iek
}-\delta_{i}^{h}K^{e}_{jek}+g_{ik}g^{fh}K^{e}_{jef}-g_{jk}g^{fh}K^{e}_{ief}\}\nonumber\\
&+&
\frac{K}{(\ell-1)(\ell-2)}(g_{ik}\delta_{j}^{h}-g_{jk}\delta_{i}^{h}).
\end{eqnarray}

Therefore we have the following
\begin{thm}\label{thm31}
$S^{h}_{ijk}=\bar{S}^{h}_{ijk}$,  namely,  $S^{h}_{ijk}$ is an
invariant under the nonholonomic connection transformation
$\nabla\rightarrow D$.
\end{thm}

 It is well
known that one of differences between sub-Riemannian geometry and
Riemannian case is that there exists a kind of singular geodesics,
which does not satisfy the geodesic equation, in sub-Riemannian
geometry, so when we consider the projective transformation of
$\nabla$, we should modify that, if semi-sub-Riemannnian connection
$D$ and sub-Riemannnian connection $\nabla$ has the same normal
geodesics, we call it the projective transformation of $\nabla$.
Therefore the Weyl projective transformation of $\nabla$ conserves
the normal geodesics invariant.

Recall the conformal curvature tensor and projective curvature
tensor (see \cite{ZJ}) of sub-Riemannian connection $\nabla$ are
respectively,
\begin{eqnarray*}
C^{h}_{ijk}&=&K^{h}_{ijk}-\frac{1}{\ell-2}\{\delta_{j}^{h}(k^{e}_{iek}-\frac{1}{\ell}K^{e}_{ike}-\frac{K}{2(\ell-1)}g_{ik})\nonumber\\
&&-\delta_{i}^{h}(K^{e}_{jek}-\frac{1}{\ell}K^{e}_{jke}-\frac{K}{2(\ell-1)}g_{jk})\nonumber\\
&&+g_{ik}(k^{e}_{jef}g^{fh}-\frac{1}{\ell}k^{e}_{jfe}g^{fh}-\frac{K}{2(\ell-1)}\delta_{j}^{h}
)\nonumber\\
&&-g_{jk}(K^{e}_{ief}g^{fh}-\frac{1}{\ell}K^{e}_{ife}g^{fh}-\frac{K}{2(\ell-1)}\delta_{i}^{h})\}\nonumber\\
&&+\frac{1}{\ell}\delta_{k}^{h}K^{e}_{ije},\\
W^{h}_{ijk}&=&K^{h}_{ijk}-\frac{1}{\ell-1}(\delta_{j}^{h}K^{e}_{iek}-\delta_{i}^{h}K^{e}_{jek}).
\end{eqnarray*}

For the semi-sub-Riemannnian connection $D$, we define the Weyl
conformal curvature tensor and the projective curvature tensor,
respectively,  by
\begin{eqnarray}
\bar{C}^{h}_{ijk}&=&R^{h}_{ijk}-\frac{1}{\ell-2}\{\delta_{j}^{h}(R^{e}_{iek}-\frac{1}{\ell}R^{e}_{ike}-\frac{R}{2(\ell-1)}g_{ik})\nonumber\\
&&-\delta_{i}^{h}(R^{e}_{jek}-\frac{1}{\ell}R^{e}_{jke}-\frac{R}{2(\ell-1)}g_{jk})\nonumber\\
&&+g_{ik}(R^{e}_{jef}g^{fh}-\frac{1}{\ell}R^{e}_{jfe}g^{fh}-\frac{R}{2(\ell-1)}\delta_{j}^{h}
)\nonumber\\
&&-g_{jk}(R^{e}_{ief}g^{fh}-\frac{1}{\ell}R^{e}_{ife}g^{fh}-\frac{R}{2(\ell-1)}\delta_{i}^{h})\}\nonumber\\
&&+\frac{1}{\ell}\delta_{k}^{h}R^{e}_{ije},\\
\bar{W}^{h}_{ijk}&=&R^{h}_{ijk}-\frac{1}{\ell-1}(\delta_{j}^{h}R^{e}_{iek}-\delta_{i}^{h}R^{e}_{jek}).
\end{eqnarray}

\begin{rem}\label{re31}
 By using $(3.7)$ and $(3.11)$, we get
\begin{eqnarray*}
\bar{C}^{h}_{ijk}&=&C^{h}_{ijk}-\frac{1}{\ell}(\delta_{j}^{h}\pi_{ik}-\delta_{i}^{h}\pi_{jk}+g_{ik}\pi_{j}^{h}-g_{jk}\pi_{i}^{h})\\
 &&-\frac{2\alpha}{\ell(\ell-2)}(\delta_{j}^{h}g_{ik}-\delta_{i}^{h}g_{jk})-\frac{\ell-2}{\ell}\delta_{k}^{h}\pi_{ij}-\frac{\alpha}{\ell}\delta_{k}^{h}g_{ij},\\
 \bar{W}^{h}_{ijk}&=&W^{h}_{ijk}+\frac{1}{\ell-1}(\delta_{j}^{h}\pi_{ik}
-\delta_{i}^{h}\pi_{jk})+(g_{ik}\pi_{j}^{h}-g_{jk}\pi_{i}^{h})\\
&& -\frac{\alpha}{\ell-1}(\delta_{j}^{h}g_{ik}
-\delta_{i}^{h}g_{jk}).
\end{eqnarray*}
Therefore unlike the Riemannian case,  here the Weyl conformal
curvature tensor $C^{h}_{ijk}$ is no longer an invariant under the
connection transformation from sub-Riemannian connection $\nabla$ to
semi-sub-Riemannnian connection $D$.
\end{rem}

Now we assume that $\bar{C}^{h}_{ijk}=C^{h}_{ijk}$, then
\begin{eqnarray*}
(\delta_{j}^{h}\pi_{ik}-\delta_{i}^{h}\pi_{jk}+g_{ik}\pi_{j}^{h}-g_{jk}\pi_{i}^{h})
+\frac{2\alpha}{\ell-2}(\delta_{j}^{h}g_{ik}-\delta_{i}^{h}g_{jk})+(\ell-2)\delta_{k}^{h}\pi_{ij}+\alpha\delta_{k}^{h}g_{ij}=0.
\end{eqnarray*}
Contracting the above equation by $k=h$, we obtain
\begin{equation}
(\ell-2)\pi_{ij}+\alpha g_{ij}=0,\nonumber\\
\end{equation}
multiplying $g^{ij}$ on both side of above equation, further we get
$\pi=0$. The inverse is also true, so we have the following result.
\begin{thm}\label{thm32}
The semi-sub-Riemannnian connection $D$ and the sub-Riemannnian
connection $\nabla$ have the same conformal curvature tensor if and
only if $\alpha$ is vanishing.
\end{thm}

Then we assume that $\bar{W}^{h}_{ijk}=W^{h}_{ijk}$, hence we have
\begin{equation}
\frac{1}{\ell-1}(\delta_{j}^{h}\pi_{ik}
-\delta_{i}^{h}\pi_{jk})+(g_{ik}\pi_{j}^{h}-g_{jk}\pi_{i}^{h})-\frac{\alpha}{\ell-1}(\delta_{j}^{h}g_{ik}
-\delta_{i}^{h}g_{jk})=0.
\end{equation}
By multiplying $g^{jk}$ in $(3.20)$, we get
\begin{eqnarray*}
\pi_{i}^{h}=\frac{\alpha}{\ell}\delta_{i}^{h}, ~\textrm{or},~
\pi_{ih}=\frac{\alpha}{\ell}g_{ih}.
\end{eqnarray*}

This implies the following
\begin{thm}\label{thm33}
The semi-sub-Riemannnian connection $D$ and the horizontal
connection $\nabla$ have the same projective curvature tensor if and
only if the characteristic tensor is proportional to a metric
tensor.
\end{thm}
\begin{proof}
We just prove the sufficiency of Theorem \ref{thm33}. Let
$\pi_{i}^{j}=\lambda \delta_{i}^{j}$, then $\pi=\pi_{i}^{i}=\lambda
\ell$, and $\pi_{ij}=\lambda g_{ij}$. Substituting  these equations
above  into the second formula in Remark 3.2,  we get
$\bar{W}^{h}_{ijk}=W^{h}_{ijk}$. This ends the proof of Theorem
\ref{thm33}.
\end{proof}
Theorem 3.3 implies the connection transformations from
sub-Riemannian connection $\nabla$ to semi-sub-Riemannnian
connection $D$  that change normal geodesics into normal geodesics
also conserve the projective curvature tensor invariant under
certain conditions.

\begin{rem}\label{re32}
By comparing the tensor $\bar{S}^{h}_{ijk}$ with the conformal
curvature tensor $\bar{C}^{h}_{ijk}$ defined by $(3.18)$, we find
that
\begin{equation}
\bar{C}^{h}_{ijk}=\bar{S}^{h}_{ijk}+\frac{1}{\ell(\ell-2)}(\delta_{j}^{h}R^{e}_{ike}
-\delta_{i}^{h}R^{e}_{jke}+g_{ik}R^{e}_{jfe}g^{fh}-g_{jk}R^{e}_{ife}g^{fh})+\frac{1}{\ell}\delta_{k}^{h}R^{e}_{ije}
\end{equation}
Given that $K^{e}_{ije}=0$, then $R^{e}_{ije}=0$(for any $i, j$), so
$\bar{C}^{h}_{ijk}=\bar{S}^{h}_{ijk}$, $C^{h}_{ijk}=S^{h}_{ijk}$.
Hence Theorem 3.1 implies that a geometric characteristic of tensor
$S^{h}_{ijk}$ is conformal invariant tensor under certain
conditions.
\end{rem}

Now we assume $R^{h}_{ijk}=K^{h}_{ijk}$, then
\begin{equation}
\delta_{j}^{h}\pi_{ik}-\delta_{i}^{h}\pi_{jk}+\pi_{j}^{h}g_{ik}-\pi_{i}^{h}g_{jk}=0.
\end{equation}
Contracting the equation $(3.22)$ with  $i$ and $h$, we get
\begin{equation}
(2-\ell)\pi_{jk}-\alpha g_{jk}=0.
\end{equation}
Multiplying the equation $(3.23)$ by $g^{jk}$ we get
\begin{equation}
2(\ell-1)\alpha=0 ,\nonumber\\
\end{equation}
and $\ell>2$, therefore $\alpha=0$; the converse is  also true, thus
we have

\begin{thm}\label{thm34}
The semi-sub-Riemannnian connection $D$ and the sub-Riemannian
connection $\nabla$ have the same Schouten curvature tensor if and
only if $\alpha$ is vanishing.
\end{thm}
A geometric characteristic of Theorem 3.4 is the connection
transformations from sub-Riemannian connection $\nabla$ to
semi-sub-Riemannnian connection $D$ conserve the Schouten curvature
tensor invariant under certain conditions.

 Now  we consider the case of $R^{h}_{ijk}=0,$ that
is, there hold
\begin{equation}
K^{h}_{ijk}=\delta_{i}^{h}\pi_{jk}-\delta_{j}^{h}\pi_{ik}+\pi_{i}^{h}g_{jk}-\pi_{j}^{h}g_{ik},
\end{equation}
let $j=h=e$, we obtain
\begin{equation}
K^{e}_{iek}=(2-\ell)\pi_{ik}-\alpha g_{ik},
\end{equation}
Multiplying the equation $(3.25)$ by $g^{ik}$ we get
\begin{equation}
K=K^{e}_{iek} g^{ik}=2(1-\ell)\alpha, \nonumber\\
\end{equation}
So we have
\begin{equation}
\alpha=\frac{K}{2(1-\ell)},
\end{equation}

Substituting $(3.26)$ into $(3.25)$, we get
\begin{equation}
\pi_{ik}=\frac{1}{2-\ell}(K^{e}_{iek}-\frac{K}{2(\ell-1)}g_{ik}),
\end{equation}
Similarly, we substitute $(3.27)$ into $(3.24)$, we have
\begin{eqnarray}\label{326}
K^{h}_{ijk}&=&-\frac{1}{\ell-2}(\delta_{i}^{h}K^{e}_{jek}-\delta_{j}^{h}K^{e}_{iek}+g_{jk}K^{e}_{ief}g^{fh}-g_{ik}K^{e}_{jef}g^{fh})
\nonumber\\
 &+& \frac{K}{(\ell-2)(\ell-1)}(g_{jk}\delta_{i}^{h}-g_{ik}\delta_{j}^{h})
\end{eqnarray}
By using $(3.17)$, equation (\ref{326}) is equivalent to
$S^{h}_{ijk}=0$. This implies the following
\begin{thm}\label{thm35}
The sub-Riemannian manifold $(M,V_{0},g)$ associated with a
semi-sub-Riemannnian connection D is flat (i.e. $R^{h}_{ijk}=0$) if
and only if the tensor $S^{h}_{ijk}$, defined by $(3.17)$,  of
sub-Riemannian connection $\nabla$ is vanishing and
$\pi_{ik}=\frac{1}{2-\ell}(K^{e}_{iek}-\frac{K}{2(\ell-1)}g_{ik})$.
\end{thm}
\begin{proof}
Here just to prove the sufficiency. If
$\pi_{ik}=\frac{1}{2-\ell}(K^{e}_{iek}-\frac{K}{2(\ell-1)}g_{ik})$,
then $\alpha=\frac{K}{2(1-\ell)}$, so
$K^{e}_{iek}=(2-\ell)\pi_{ik}-\alpha g_{ik}$, and
\begin{equation}\label{2}
R^{e}_{iek}=K^{e}_{iek}+(\ell-2)\pi_{ik}+\alpha g_{ik}=0\nonumber\\
\end{equation}
By the first Bianchi identity, we know
\begin{equation}\label{2}
R^{e}_{ike}=R^{e}_{kei}-R^{e}_{iek}=0,\nonumber\\
\end{equation}
and
\begin{eqnarray*}
\bar{C}^{h}_{ijk}&=&\bar{S}^{h}_{ijk}+\frac{1}{\ell(\ell-2)}(\delta_{j}^{h}R^{e}_{ike}
-\delta_{i}^{h}R^{e}_{jke}+g_{ik}R^{e}_{jfe}g^{fh}-g_{jk}R^{e}_{ife}g^{fh})+\frac{1}{\ell}\delta_{k}^{h}R^{e}_{ije}\\
&=&S^{h}_{ijk}\\
&=&0
\end{eqnarray*}
 Therefore, we have
\begin{eqnarray*}
R^{h}_{ijk}&=&\bar{C}^{h}_{ijk}+\frac{1}{\ell-2}\{\delta_{j}^{h}(R^{e}_{iek}-\frac{1}{\ell}R^{e}_{ike}-\frac{R}{2(\ell-1)}g_{ik})\\
&&+\delta_{i}^{h}(R^{e}_{jek}-\frac{1}{\ell}R^{e}_{jke}-\frac{R}{2(\ell-1)}g_{jk})\\
&&-g_{ik}(R^{e}_{jef}g^{fh}-\frac{1}{\ell}R^{e}_{jfe}g^{fh}-\frac{R}{2(\ell-1)}\delta_{j}^{h}
)\\
&&+g_{jk}(R^{e}_{ief}g^{fh}-\frac{1}{\ell}R^{e}_{ife}g^{fh}-\frac{R}{2(\ell-1)}\delta_{i}^{h})\}\\
&&-\frac{1}{\ell}\delta_{k}^{h}R^{e}_{ije}\\
&=&0.
\end{eqnarray*}
This completes the proof of Theorem \ref{thm35}.
\end{proof}

We now assume that, for any $X,Y,Z\in V_{0}$, there are
\begin{eqnarray*}
&&R(X,Y)Z=0,\\
&&(\nabla_{X}T)(Y,Z)=0.
\end{eqnarray*}
A manifold satisfying these two conditions is called a group
manifold with respect to $\nabla$.
\begin{examp}
Carnot group G is a group manifold with respect to $\nabla$ defined
by $(2.9)$.
\end{examp}

 In fact, let $X=X^{i}X_{i}, Y=Y^{j}X_{j},Z=Z^{k}X_{k}$, and by $(2.9)$
and $(2.11)$, then the horizontal curvature tensor can be given
exactly as
\begin{eqnarray*}
K(X,Y)Z&=&\nabla_{X}\nabla_{Y}Z-\nabla_{Y}\nabla_{X}Z-\nabla_{[X,Y]_{0}}Z\\
&=&X^{i}X_{i}(Y^{j})X_{j}(Z^{k})X_{k}+Y^{j}X^{i}X_{i}X_{j}(Z^{k})X_{k}-Y^{j}X_{j}(X^{i})X_{i}(Z^{k})X_{k}\\
&&-Y^{j}X^{i}X_{j}X_{i}(Z^{k})X_{k}-X^{i}X_{i}(Y^{j})X_{j}(Z^{k})X_{k}+Y^{j}X_{j}(X^{i})X_{i}(Z^{k})X_{k}\\
&=&0,
\end{eqnarray*}
On the other hand, the horizontal torsion tensor of horizontal
vector fields of Z,Y is
\begin{eqnarray*}
T(Y,Z)&=&\nabla_{Y}Z-\nabla_{Z}Y-[Y,Z]_{0}\\
&=&Y^{j}X_{j}(Z^{k})X_{k}-Z^{k}X_{k}(Y^{j})X_{j}-Y^{j}X_{j}(Z^{k})X_{k}+Z^{k}X_{k}(Y^{j})X_{j}=0
\end{eqnarray*}
\begin{eqnarray*}
T(\nabla_{X}Y,Z)&=&\nabla_{\nabla_{X}Y}Z-\nabla_{Z}\nabla_{X}Y-[\nabla_{X}Y,Z]_{0}\\
&=&X(Y^{j})\nabla_{X_{j}}Z-Z(X(Y^{j}))X_{j}-X(Y^{j})\nabla_{Z}X_{j}-X(Y^{j})X_{j}(Z^{k})X_{k}\\
&&+Z(X(Y^{j}))X_{j}-X(Y_{j})[X_{j},Z]_{0}\\
&=&0,
\end{eqnarray*}
so one has
\begin{eqnarray*}
(\nabla_{X}T)(Y,Z)=\nabla_{X}T(Y,Z)-T(\nabla_{X}Y,Z)-T(Y,\nabla_{X}Z)=0.
\end{eqnarray*}

If $D$ is a semi-sub-Riemannian connection, then we have
\begin{equation}
(D_{X}T)(Y,Z)=0\rightleftharpoons (D_{X}\pi)(Z)Y-(D_{X}\pi)(Y)Z=0,
\end{equation}
In a local frame $\{e_{i}\}$, by taking $X=e_{i},Y=e_{j},Z=e_{k}$,
then we get
\begin{eqnarray*}\label{211}
0&=&\{D_{e_{i}}\pi_{k}-\pi(D_{e_{i}}e_{k})\}e_{j}-\{D_{e_{i}}\pi_{j}-\pi(D_{e_{i}}e_{j})\}e_{k}
\nonumber\\
 &=& \{e_{i}(\pi_{k})-\Gamma_{ik}^{e}\pi_{e}\}e_{j}-
 \{e_{i}(\pi_{j})-\Gamma_{ij}^{e}\pi_{e}\}e_{k},
\nonumber\\
\end{eqnarray*}
Thus we know
\begin{equation}
e_{i}(\pi_{j})-\Gamma_{ij}^{e}\pi_{e}=0,
\end{equation}
substituting $(3.4)$ into $(3.30)$ and using $(3.8)$ we deduce
\begin{equation}
\pi_{ij}=-\frac{1}{2}g_{ij}\pi_{e}\pi^{e},
\end{equation}
By virtue of Theorem 3.3, we have $W^{h}_{ijk}=\bar{W}^{h}_{ijk}$.
Since $\bar{R}^{h}_{ijk}=0$, we get $\bar{R}^{e}_{iek}=0$, then
$\bar{W}^{h}_{ijk}=0$, so $W^{h}_{ijk}=0$. This implies the
following
\begin{prop}
If sub-Riemnnian manifold $(M,V_{0},g)$ is a group manifold with
respect to the semi-sub-Riemannian connection $D$, then $M$ is
projective flat.
\end{prop}
Then substituting $(3.31)$ into $(3.24)$ we get
\begin{equation}
K^{h}_{ijk}=\pi_{e}\pi^{e}(\delta_{j}^{h}g_{ik}-\delta_{i}^{h}g_{jk}).\nonumber\\
\end{equation}
It is not hard to see by a direct checking up on a few things that
the converse is also true, hence we obtain
\begin{thm}\label{thm36}
A sub-Riemannian manifold $(M,V_{0},g)$ with vanishing curvature
with respect to semi-sub-Riemannnian connection $D$ is a group
manifold if and only if $M$ is of constant curvature.
\end{thm}

\section{ Acknowledgments}

I would like to thank Professor P.B. Zhao who encouraged me to study
the semi-symmetric connection on sub-Riemannian manifolds, for very
useful discussions and valuable suggestion. Both authors would like
to thank Professor X.P. Yang for his encouragement and help!

\end{document}